\documentclass[letterpaper]{amsart}
\usepackage{tikz-cd}
\usepackage{amsmath}
\usepackage{mathtools}
\usepackage{amsmath,amssymb,amsthm,mathrsfs}
\usepackage[utf8]{inputenc}
\usepackage{hyperref}
\usepackage{enumitem}

\numberwithin{equation}{section}
\theoremstyle{plain} 
\newtheorem{theorem}{Theorem}
\numberwithin{theorem}{section}
\newtheorem*{theorem*}{Theorem}
\newtheorem{prop}[theorem]{Proposition}
\newtheorem{lemma}[theorem]{Lemma}
\newtheorem{coro}[theorem]{Corollary}
\newtheorem{conj}[theorem]{Conjecture}
\theoremstyle{definition}
\newtheorem{definition}[theorem]{Definition}
\newtheorem*{claim}{Claim} 
\newtheorem{example}[theorem]{Example}
\theoremstyle{remark}
\newtheorem{remark}[theorem]{Remark}
\theoremstyle{definition}

\newtheorem{thmx}{Theorem}

\DeclarePairedDelimiter\ceil{\lceil}{\rceil}
\DeclarePairedDelimiter\floor{\lfloor}{\rfloor}

\newcommand{\bbC}{\mathbb{C}}

\newcommand{\Q}{\mathbb{Q}}

\newcommand{\cA}{\mathcal{A}}
\newcommand{\cB}{\mathcal{B}}

\newcommand{\cG}{\mathcal{G}}
\newcommand{\cH}{\mathcal{H}}

\newcommand{\cJ}{\mathcal{J}}
\newcommand{\cK}{\mathcal{K}}
\newcommand{\cL}{\mathcal{L}}
\newcommand{\cM}{\mathcal{M}}

\newcommand{\cO}{\mathscr{O}}

\newcommand{\rarr}{\longrightarrow}

\newcommand{\sL}{\mathscr{L}}

\newcommand{\sF}{\mathscr{F}}
\newcommand{\sE}{\mathscr{E}}

\newcommand{\sO}{\mathscr{O}}
\newcommand{\sA}{\mathscr{A}}
\newcommand{\shTA}{\mathscr{T}}
\newcommand{\sD}{\mathscr{D}}

\newcommand{\Var}{\textup{Var}}

\DeclareMathOperator{\supp}{supp}
\DeclareMathOperator{\red}{red}
\DeclareMathOperator{\gr}{gr}

\mathchardef\ordinarycolon\mathcode`\:
\mathcode`\:=\string"8000
\begingroup \catcode `\:=\active
  \gdef:{\mathrel{\mathop\ordinarycolon}}
\endgroup

\title{Hyperbolicity for log smooth families with maximal variation}

\author{Chuanhao Wei}
\address{Department of Mathematics, Stony Brook University, 100 Nicolls Rd, Stony Brook, NY 11794, USA}
\email{\tt chuanhao.wei@stonybrook.edu}

\author{Lei Wu}
\address{Department of Mathematics, University of Utah,
155 S 1400 E,  Salt Lake City, UT 84112, USA}
\email{\tt lwu@math.utah.edu}

\begin{document}

\begin{abstract}
We prove that the base space of a log smooth family of log canonical pairs of log general type is of log general type as well as algebraically degenerate, when the family admits a relative good minimal model over a Zariski open subset of the base and the relative log canonical model is of maximal variation. 
\end{abstract}
\maketitle
\section{Introduction}
In this article, we study hyperbolic properties of the base spaces of families of log canonical pairs. Recall that a smooth quasi-projective variety $V$ is of log general type, if it has a log smooth projective compactification $(Y, D)$, with the log canonical bundle $\omega_Y(D)$ big. More generally, a quasi-projective variety is called of log general type, if it is so after a desingularization.
Our main result is the following.
\begin{thmx}\label{thm:main}
Let $f_V\colon (U, D_U)\to V$ be a log smooth family of log canonical pairs over a smooth quasi-projective variety $V$, with general fibers being of log general type. If $f_V$ admits a relative good minimal model over a Zariski open subset of $V$ and it is of maximal variation, then 
\begin{enumerate}
    \item $V$ is of log general type; 
    \item every holomorphic map $\gamma\colon \bbC\to V$ is algebraically degenerate, that is, the image of $\gamma$ is not Zariski dense.   
\end{enumerate}
\end{thmx}

When $D_U$ is trivial and more generally the pair $(U, D_U)$ is further assumed to be Kawamata log terminal, Theorem \ref{thm:main} applies, as the assumption on the existence of good minimal models is guaranteed by \cite{BCHM}. 


The conjecture of Green-Griffiths-Lang \cite{GG80, Lan86} predicts that for a log general type variety, every holomorphic map from $\bbC$ to the variety is algebraically degenerate. As a consequence of Theorem \ref{thm:main}, the conjecture holds for the base space $V$.


In this article, the main technical tool we develop to prove the above main result is the existence of the Higgs sheaf on the base, as in the following theorem. It not only naturally generalizes the Higgs sheaf constructed in \cite{PS17} to the log family case, but also has been constructed on the whole base space, up to a birational base change. It is essential in proving Theorem \ref{thm:main} (2).
\begin{thmx}\label{thm:refineHiggs}
Let $f_V\colon (U, D_U)\to V$ be a log smooth family of log general type klt pairs with maximal variation. After replacing $V$ by a birational model, there exist a smooth projective compactification $Y$ of $V$ with $E=Y\setminus V$ a simple normal crossing divisor, a nef and big line bundle $\cL$ and a graded $\sA_Y^\bullet(-\log E)$-module $(\sF_\bullet,\theta_\bullet)$ together with  an inclusion of graded $\sA_Y^\bullet(-\log (E+S))$-modules
\[(\sF_\bullet,\theta_\bullet)\subseteq (\sE_\bullet,\theta_\bullet)\]
satisfying
\begin{enumerate}[label=(\alph*)]
\item$\sF_k$ is reflexive for each $k\ge0$ and $\cL\subseteq \sF_0$.
\item $(\sE_\bullet,\theta_\bullet)$ is the Higgs bundle associated to the Deligne extension with eigenvalues in $[0,1)$ of a geometric VHS defined outside of a simple normal crossing divisor $E+S$.
\end{enumerate}
\end{thmx}

\subsection*{Previous results and conjectures}
Theorem \ref{thm:main} (1) in the case of families of canonically polarized manifolds, is known as the Viehweg hyperbolicity conjecture and was first proved by Campana and P\u aun \cite{CP15} based on the so-called Viehweg-Zuo sheaves constructed in \cite{VZ03}. Popa and Schnell \cite{PS17} generalized the construction for the families of projective varieties with the geometric generic fibers admitting good minimal models using Saito's theory of Hodge modules, and proved the Viehweg hyperbolicity conjecture also for these families, in particular, Theorem \ref{thm:main} (1) in the case of families of  smooth projective varieties of general type.  

Theorem \ref{thm:main} (2) in the case of families of canonically polarized manifolds, was proved by Viehweg and Zuo. They used it as an intermediate step proving Brody hyperbolicty of the moduli stack of canonically polarized manifolds \cite[Theorem 0.1]{VZ03}. In the case of smooth families of minimal varieties of general type, it was proved by Popa, Taji and the second author \cite{PTW18}. They also proved the algebraic degeneracy for families with geometric generic fibers admitting good minimal models, when the base spaces are surfaces. 

Based on \cite{VZ03} and \cite{PTW18}, the algebraic degeneracy for smooth families of varieties with semi-ample canonical bundles was recently proved by Deng \cite{Den18b}, by proving a result of a Torelli-type property for certain geometric variations of Hodge structures (VHS); see also Theorem \ref{th:Deng}. Using it, the author proved Brody hyperbolicity of moduli stacks of polarized monifolds with semi-ample canonical bundles. To and Yeung \cite{TY15} proved that the base spaces of effective parametrized (a more restrictive condition replacing maximal variation) families of canonically polarized manifolds are Kobayashi hyperbolic, from a differential geometric point of view; see also Schumacher \cite{Sch17} and Berndtsson, P\u aun and Wang \cite{BPW17}. Deng \cite{Den18} generalizes the Kobayashi hyperbolicity result to smooth families of minimal varieties of general type, where the author combined the Hodge theoretical method in \cite{VZ03} and its refinement in \cite{PTW18}, together with differential geometric methods. See also the remarks in \S \ref{scn:kobayashi}, about how the differential geometric methods could be applied in our situation.

Kebekus and Kov\'acs \cite{KK08} proposed a conjecture generalizing the Viehweg hyperbolicity conjecture to families of varieties with arbitrary variation. We propose a pair version here as the first conclusion in our main theorem serves as an evidence. See also \cite{W17b} for other evidences over $\bbC^*$ and abelian varieties. 

\begin{conj}
Given a log smooth family $f:(U,D)\to V$ with $V$ quasi-projective, coefficients of $D$ in $(0,1)$, and $(U_y,D_y)$ of log general type for all $y\in V$, then either $\kappa(V)=-\infty$ and $\Var(f)<\dim V$, or $\Var(f)\leq \kappa(V)$.
\end{conj}




\subsection*{Structure} We first state two applications in Section \ref{sect:applications}. One is about the Brody hyperbolicity of moduli stack of stable log-smooth pairs; the other is about the simultaneous resolution of stable families of log pairs.  In Section \ref{sect:variation}, we recall the definition of the variation for families of log canonical pairs of log general type, and discuss how it changes under perturbation of coefficients when assuming MMP, which helps us reduce the proof of our main theorem to the case for families of klt pairs. In Section \ref{sect:stred}, we set up stable reductions for log smooth families, by using moduli stacks of stable log varieties, a replacement of mild reductions in \cite{VZ02, VZ03}. In Section \ref{sect:genfree},  we deduce generic freeness of direct images of pluri-canonical sheaves in the pair case from using Viehweg's fiber-product tricks and the stable reduction together with an analytic extension theorem. The generic freeness provides enough sections to construct Higgs sheaves. We generalize the Higgs sheaf construction of Popa and Schnell \cite{PS17} to the pair case for log smooth families in Section \ref{sect:higgs}. We conclude this paper, in Section \ref{sect:proofmain}, by showing the proof of Theorem \ref{thm:main}. 

\subsection*{Notations and Conventions}
We mainly follow the notations and the terminology from \cite{PTW18} and \cite{KP16}, and work over the field of complex numbers $\bbC$. For a family of log pairs, we mean a projective surjective morphism $f\colon (X,D)\to Y$ with $Y$ being normal, $(X, D)$ being a log pair and the support of $D$ being either trivial or dominant $Y$, and general fibers of $f$ being connected. We say that the family of log pairs $f$ is a semi-stable family if it is flat and all fibers are semi log canonical (slc) pairs, and the relative log canonical divisor $K_{(X,D)/Y}$ is $\Q$-Cartier and stable under arbitrary base-changes, where $K_{(X,D)/Y}=K_{X/Y}+D$. It is called a stable family, if moreover, all the fibers are stable slc pairs. We say that $f$ is a log smooth family if $(X,D)$ is a log smooth pair with both $X$ and every stratum of the support of $D$ are smooth over the base (this implies the base being smooth) and with fibers projective. Abusing the notation, we denote $\omega^m_{X/Y}(mD):=\cO_X(mK_{(X,D)/Y)})$, as long as $mK_{(X,D)/Y}$ is Cartier. 

When $(X, D)$ is a smooth log pair, we use $\sD_X$ to denote the sheaf of differential operators with algebraic (or holomorphic) function coefficients and $\sD_X(-\log D)$ the subalgebra of $\sD_X$ generated by $\sO_X$ and $\shTA_X(-\log D)$, the sheaf of vector fields with logarithmic zeros along the support of $D$. We also write graded algebras
$$\sA^\bullet_X(-\log D)=\textup{Sym}^\bullet(\shTA_X(-\log D))\subseteq \sA^\bullet_X=\textup{Sym}^\bullet(\shTA_X).$$
With the order filtration of differential operators, we identify 
$$\sA^\bullet_X(-\log D)\simeq \gr^F_\bullet\sD_X(-\log D) \textup{ and }\sA^\bullet_X\simeq \gr^F_\bullet\sD_X.$$ 

\subsection*{Acknowledgement.} 
The authors are deeply grateful to Christopher Hacon, for many useful discussions and suggestions and for the ideas of the proof of Lemma \ref{L:compare var}. The authors also thank Junyan Cao, Jingjun Han, S\'andor Kov\'acs, Mihnea Popa, Karl Schwede, Christian Schnell and Ziwen Zhu, for useful discussions and for answering questions. They thank Kang Zuo for telling them Parshin's tricks.

\section{Some applications} \label{sect:applications}
\subsection{Brody hyperbolicity of moduli stack}
Applying Theorem \ref{thm:main} to families of canonically polarized log smooth pairs with boundary coefficients $>1/2$, one can prove that the log-smooth locus of $\cM_{n,v,I}$, the moduli stack of stable log varieties with fixed dimension
$n$, volume $v$, and coefficient set $I$, as defined in \cite[Definition 6.6]{KP16}, is of log general type and Brody hyperbolic. This generalizes \cite[Theorem 0.1]{VZ03} to the pair case. Note that the requirement of the boundary coefficients is to guarantee that the family is log smooth as long as each fiber is log smooth. See Example \ref{ex:easy} for a counterexample without this assumption. We denote the corresponding coarse moduli space by $M_{n,v,I}$.
\begin{theorem}\label{thm:mgn}
Let $Y$ be a quasi-projective variety and let  $f\in \cM_{n,v,I}(Y)$ be a stable family of log-smooth pairs with dimension
$n$, volume $v$, and coefficient set $I$. Assume that all elements in $I$ are $>1/2$. If the induced moduli map $Y\to M_{n,v,I}$ is quasi-finite over its image, then the base $Y$ satisfies
\begin{enumerate}
    \item every subvariety of $Y$ is of log general type;
    \item $Y$ is Brody hyperbolic, that is, there are no nonconstant holomorphic map $\gamma\colon\bbC\to Y$.  
\end{enumerate}
\end{theorem}

\begin{proof}
Assume that we have a subvariety $V\subset Y$. Blowing-up the singularities of $V$, we have an induced stable family of log-smooth pairs $f_V\colon (U, D_U)\to V$. With the assumption on the coefficients of $D_U$, it is not hard to see that $f$ is log-smooth. Then the first statement follows from Theorem \ref{thm:main} (1). 

For the second one, if $\gamma\colon \bbC\rarr Y$ is a nonconstant holomorphic map, then we take $V$ to be the Zariski closure of the image of $\gamma$ and hence $\gamma$ has Zariski dense image in $V$. By the lifting property of $\gamma$ through birational morphisms, we are free to assume $V$ being smooth. This contradicts to Theorem \ref{thm:main} (2), because the induced family over $V$ is of maximal variation, as the original moduli map is quasi-finite.
\end{proof}
From the above theorem, in particular, the moduli stack of Riemann surfaces of genus $g$ with $n$ marked points $\cM_{g,n}$, is of log general type and Brody hyperbolic. 
In the case of moduli stack for Riemann surfaces $\cM_g$, the first conclusion above is an easy consequence of Viehweg's hyperbolicity conjecture; the second is a result implied by \cite{Ahl}, \cite{Roy} and \cite{Wol} from a differential geometric point of views using the negativity of the holomorphic sectional curvature of the Weil-Petersson metric on the Teichm\"uller spaces $T_g$. 
The hyperbolicity of $\cM_{g,n}$ can also be proved  by using hyperbolicity of $\cM_g$ and Parshin's covering tricks.

It is well-known that $M_{g,n}$ does not carry a universal family in general, for which the above theorem is not always applying. However, by \cite[Th\'eor\`em 16.6]{LMB}, as $\cM_{g,n}$ is a Deligne-Mumford stack, there exists a finite covering $S\to M_{g,n}$ induced by a universal family in $\cM_{g,n}(S)$ for every $n$ and $g$. Therefore, the above theorem specifically implies that whenever $3g-3+n>0$, $S$ is Brody hyperbolic and every subvariety of $S$ is of log general type. When considering the moduli spaces, based on the works of Eisenbud, Harris and Mumford, Logan \cite{Log} proved that $\overline M_{g,n}$ are of general type, for almost all $g$ and $n$.

The following example tells us that the log smooth assumption on the family in Theorem \ref{thm:main} can hardly be weakened in general.

\begin{example}\label{ex:easy}
Consider the following family defined by the natural projection onto the second factor $p_2:(\mathbb{P}^1\times \mathbb{P}^1,L)\to \mathbb{P}^1$, where the boundary divisor $L$ consists of $3$ different trivial sections with the coefficient $\frac{2}{3}$, and a diagonal section with the coefficient $\frac{1}{4}$. One checks that it is a stable family, with maximal variation, and all the geometric fibers are stable, klt, log smooth pairs. However, not all strata are relatively smooth, i.e. $p_2$ is not a log smooth family. Meanwhile, the base is $\mathbb P^1$, which is on the opposite of being hyperbolic.
\end{example}

\subsection{Simultaneous resolutions of log pairs}
Another application for Theorem \ref{thm:main} is about existence of the simultaneous resolutions of log pairs.

\begin{definition}
Given a semi-stable family of log canonical pairs $f:(U,D)\to V$ with $V$ smooth quasi-projective, we say that $f$ admits a \emph{simultaneous resolution}, if there is a log-resolution $\phi:(\tilde U, \tilde D)\to (U,D)$ over $V$, where $\tilde D$ is defined by the relation
$$K_{\tilde U}+\tilde D=\phi^*(K_{U}+D),$$
such that the induced family $g:(\tilde U, \tilde D)\to V$ is log-smooth.
\end{definition}
\begin{theorem}
Given a stable family of log canonical pairs $f:(U,D)\to V$, with $V$ smooth quasi-projective and $f$ of maximal variation, if $V$ is not of log general type or there is a holomorphic map $\gamma:\bbC\to V$ with a Zariski-dense image, then $f$ does not admit any simultaneous resolution.
\end{theorem}
\begin{proof}
Assume that we have a simultaneous resolution $\phi:(\tilde U, \tilde D)\to (U,D)$ over $V$, then the induced $g:(\tilde U, \tilde D)\to V$ satisfies all assumptions in Theorem $A$.
\end{proof}

\section{Variation for families of log canonical pairs}\label{sect:variation}
The birational variation of families was first introduced by Viehweg \cite{Vie83} when studying Iitaka's $C_{n,m}$ conjecture; see also \cite{Kaw85} for an equivalent definition. It measures how general fibers of the family vary up to birational isomorphism classes. 

In the pair case, Kov\'acs and Patakfalvi \cite[Definition 6.16 and 9.3]{KP16} defined variation for stable families and more generally for families with general fiber being log canonical and of log general type, when assuming the existence of log canonical models. In the case of stable families, it is the dimension of the image of the base in a certain moduli stack under the moduli map. In particularly, it measures how the isomorphic classes of fibers vary in family. In the case for families $f\colon (X,D) \to Y$ with general fibers being log canonical and of log general type, when assuming existence of the relative log canonical model over a Zariski open subset $Y_0\subseteq Y$, calling it $f_c$, the variation of $f$ is defined as 
\[\Var(f)=\inf\{\Var(f_c|_{U})\mid \textup{Zariski open } U\subseteq Y_0\}.\]
In particular, if the family $f$ has maximal variation (that is, $\Var(f)=\dim Y$), then $\Var(f)=\Var(f_c)$. This means that in the maximal variation case, the $\inf$ in the definition is redundant.  

When general fibers are Kawamata log terminal (klt) pairs, since the log canonical model exists in this case by \cite{BCHM}, the assumption is always fullfilled. In particular, the variation is defined for families with general fibers being klt pairs of log general type. 

It is worth mentioning that when considering families with general fibers being general type smooth projective varieties, the variation in the sense of Kov\'acs and Patakfalvi is equivalent to the original definition by Viehweg, if one applies \cite[Corollary 6.20]{KP16}.

When assuming the relative MMP, we obtain the following lemma regarding how the variation changes, after perturbing coefficients. This will play an important rule in proving Theorem \ref{thm:main}.  

\begin{lemma}\label{L:compare var}
Let $f:(X,D)\to Y$ be a family of log pairs with $X$ being $\Q$-Gorenstein and $(X,D)$ log canonical. Assume that for all $0< \epsilon \ll 1$,  fibers of f are of log general type over a Zariski open subset $Y_0\subseteq Y$, fibers of $f^\epsilon:(X,D^\epsilon)\coloneqq(X, D-\epsilon D)\to Y$ are klt pairs of log general type over $Y_0$ and f admits a relative good minimal model over $Y_0$, then we have
$$\Var(f)\le \Var(f^{\epsilon_0}),$$
for some $0<\epsilon_0\ll1$.
\end{lemma}
\begin{proof}
For simplicity of notations, we replace $Y$ by $Y_0$.
Consider the relative minimal model of $f$, $f_m:(X_m,D_m)\to Y$, with the induced morphism $\phi:X\dashrightarrow X_m$. We can find $0<\epsilon\ll 1$, such that $\phi$ is also a part of running the minimal model program of $(X, D^\epsilon)$, and we denote $D^\epsilon_m:=\phi_*  D^\epsilon$. Let $(X_c, D_c)\to Y$ be the relative log canonical model of $f_m:(X_m,D_m)\to Y$, and let $\psi: X_m\to X_c$ be  the induced morphism. Now we consider the relative log canonical model of $(X_m, D^\epsilon_m)$ with respect to $\psi$, and denote it by $\alpha:(X_c^\epsilon, D^\epsilon_c)\to X_c$. We claim that
\begin{claim}
There exist $0<\epsilon_0\ll1$, so that for every  $\epsilon\ge \epsilon_0$ sufficiently close to $\epsilon_0$, $D^\epsilon_c$ is $\Q$-Cartier and the induced morphism $f^{\epsilon}_c:(X_c^\epsilon, D^\epsilon_c)\to Y$ in the following diagram is actually the relative log canonical model of $f^{\epsilon}$
\end{claim}

\begin{center}
\begin{tikzcd}
(X,D^\epsilon)\arrow{d}{id}\arrow[dashed]{r}{\phi}\arrow[bend right=37]{dd}[swap]{f^\epsilon}
& (X_m, D^\epsilon_m)\arrow{d}{id}\arrow[dashed]{r}{\psi^\epsilon}
& (X_c^\epsilon,D_c^\epsilon)\arrow{d}{\alpha}\\
(X,D)\arrow{d}{f}\arrow[dashed]{r}{\phi}
& (X_m, D_m)\arrow{ld}[swap]{f_m}\arrow{r}{\psi}
& (X_c,D_c)\arrow{lld}{f_c}\\
Y. & & 
\end{tikzcd}
\end{center}

Assuming the claim, we prove the inequality $\Var(f)\le \Var(f^{\epsilon_0})$. By definition, it suffices to show that, if two fibers over general $y_1,y_2\in Y$ of $f_c^{\epsilon_0}$ are isomorphic, then, those two fibers over $y_1$ and $y_2$ for $f_c$ are also isomorphic. Hence, it is enough to show that the family $f_c$ can be recovered by the family $f_c^{\epsilon_0}$. 
To prove this, we write $D'_c=\frac{1}{1-\epsilon_0}D_c^{\epsilon_0}$.
We have the induced birational map over $Y$, $(X_m, D_m) \dashrightarrow (X^{\epsilon_0}_c, D'_c)$, which factors through $\psi$. Due to the definition of $\psi^{\epsilon_0}$, it is a morphism away from a codimension at least 2 subset, and it only contracts the curve classes that are contracted by $\psi$. Hence, it preserves the relative log canonical ring over $Y$. This means that $f_c$ is the relative log canonical model of $(X_c^{\epsilon_0},D'_c)\to Y$ and we recover $f_c$ from $f_c^{\epsilon_0}$

Finally, we prove the claim. By the klt assumption on $(X, D^\epsilon)$, $(X_m, D_m^\epsilon)$ is also klt for all $0<\epsilon\ll1$. Hence, there exists $0<\epsilon_0\ll1$ so that for all $\epsilon$ sufficiently close to $\epsilon_0$, $(X_m, D_m^\epsilon)$ share the same log canonical model over $X_c$, thanks to finiteness of models \cite[Theorem E]{BCHM}. In particular, $D^{\epsilon}_c$ is $\Q$-Cartier. We then prove that for every $\epsilon\ge \epsilon_0$ sufficiently close to $\epsilon_0$, $K_{Y_c^\epsilon}+D_c^\epsilon$ is ample over $Y$. Pick a curve class $\xi$ contracted by $f_c^\epsilon$.
 If $\xi$ is contracted by $\alpha$, then obviously $\xi\cdot (K_{Y_c^\epsilon}+D_c^\epsilon)>0$. Otherwise, we may assume that $\xi$ is not contracted by $\alpha$ and $\xi\cdot (K_{X^\epsilon_c}+D_c^\epsilon)\leq 0$. Note that, as discussed above, $f_c$ is also the relative log canonical model of $(X^\epsilon_c, \frac{1}{1-\epsilon}D^\epsilon_c)$ over $Y$. Hence, we have $\xi\cdot (K_{X^\epsilon_c}+\frac{1}{1-\epsilon}D^\epsilon_c))> 0$. We then conclude that $\xi \cdot D^\epsilon_c>0$ and $\xi\cdot K_{X^\epsilon_c}<0$. They imply that, for $0<\delta\ll1$ we have $\xi\cdot (K_{X^\epsilon_c}+D^{\epsilon+\delta}_c)< 0$. By the cone theorem, we can find a negative extremal ray $\eta$ of $K_{X^\epsilon_c}+D^{\epsilon+\delta}_c$, that is contracted by $f^\epsilon_c$, but not by $\alpha$. Repeating the argument for proving $\xi\cdot K_{X^\epsilon_c}<0$, we have $\eta\cdot (K_{X^\epsilon_c}+\frac{1}{1-\epsilon-\delta}D^{\epsilon+\delta}_c)>0$ and $\eta\cdot K_{X^\epsilon_c}<0.$
Hence, on the one hand, for a fixed positive integer $k$, such that $k(K_{X^\epsilon_c}+\frac{1}{1-\epsilon-\delta}D^{\epsilon+\delta}_c)$ is Cartier, we have
$$(k+1)\eta\cdot (K_{X^\epsilon_c}+\frac{1}{1-\epsilon-\delta}D^{\epsilon+\delta}_c)>1;$$
 on the other hand, by \cite[Theorem 1]{Ka91}, we have
$$\frac{1}{2n}\eta \cdot K_{X^\epsilon_c}\geq -1,$$
where $n=\dim X$.
Adding the two inequalities, we have
\begin{equation} \label{eq:eta111}
 \eta \cdot (K_{X^\epsilon_c}+\frac{k+1}{(k+1+1/2n)(1-\epsilon-\delta)}D^{\epsilon+\delta}_c)>0.   
\end{equation}
Since $\epsilon_0\ll1$, we know $\frac{k+1}{(k+1+1/2n)(1-\epsilon-\delta)}<1$. Therefore, the inequality \eqref{eq:eta111} contradicts to the assumption that $\eta$ is a negative extremal ray.
\end{proof}
As illustrated in the following example, strict inequality in the above lemma might happen. 
\begin{example}\label{ex:main}
We start with a trivial family $p_2: (\mathbb{P}^2\times \mathbb{P}^1, \tilde A)\to \mathbb{P}^1$, where $\tilde A=p_1^* A$, for $A$ being four lines on $\mathbb{P}^2$ at generic location, and $p_1:\mathbb{P}^2\times \mathbb{P}^1\to \mathbb{P}^2$ the natural projection. Pick one line $L$ in $A$ and consider the diagonal section $S$, $\mathbb{P}^1\to p_1^*L$. Now we blow up the section $S$ on $\mathbb{P}^2\times \mathbb{P}^1$, we get $B:= \text{Bl}_S (\mathbb{P}^2\times \mathbb{P}^1)$, with the exceptional divisor $E$. We have an induced family $f:(B,\tilde A'+E)\to \mathbb{P}^1$, with $\tilde A'$ being the strict transform of $\tilde A$. Now, each fiber of $f$ is just blowing up one point on $L\subset \mathbb{P}^2$. 
It is obvious that the relative log canonical model of $f$ is just our starting trivial family $p_2$. Hence $\Var(f)=0$. However, if we perturb the coefficient of the exceptional divisor, considering $f^\epsilon: (B, \tilde A'+(1-\epsilon)E)\to \mathbb{P}^1$, with $0<\epsilon< 1$, one checks easily that $f^\epsilon$ itself is a stable family over $\mathbb P^1$ away from three points. However, general fibers of $f^\epsilon$ are not isomorphic by counting parameters, which implies that $\Var(f^\epsilon)=1$. 

\end{example}

\section{Stable reduction}\label{sect:stred}

\subsection*{Singularities of families of SLC pairs}
We discuss singularities of families of slc pairs. 

\begin{prop}\label{prop:singtotal}
Let $f:(X,D)\to Y$ be a semi-stable family, with general fibers being log canonical (resp. klt) pairs and $Y$ being smooth. Then $(X,D)$ is also a log canonical (resp. klt) pair. 
\end{prop}
\begin{proof}
Observe that $X$ is normal (\cite[Theorem 23.9]{Mat89}), and $K_X+D$ is $\mathbb{Q}$-Cartier by assumption. We first prove the klt case by induction on the dimension of $Y$. We assume that $Y$ is a smooth curve around the origin $0$, with $(X_0, D_0)$ the only slc fiber. By inversion of adjunction, we see that $(X,D)$ is klt away from $X_0$. Take a log resolution of $(X,D+X_0)$, $\pi\colon X'\to X$. Denote $D'$ and $X'_0$, the strict transform of $D$ and $X_0$ on $X'$ respectively.
We write
\begin{equation}\label{E:pullback K_X+D}
    K_{X'}+D'+\sum_i a_iE_i=\pi^*(K_X+D),
\end{equation}
where $E_i$'s are exceptional,
and $\pi^*X_0=X'_0+\sum a'_iE_i$. Since $(X,D)$ is klt away from $X_0$, if $E_i$ is not a divisor over centers contained in $X'_0$, then  $a_i<1$.

Applying \eqref{E:pullback K_X+D}, by adjunction we have 
\[
    K_{X'_0}+D'|_{X'_0}=\pi^*(K_{X_0}+D_0)-(\sum_i(a_i+a'_i)E_i)|_{X'_0}.
\]
Due to the assumption that $(X_0, D_0)$ is slc, we have $(a_i+a'_i)\leq 1$, for those $E_i$ over $X'_0$. Thanks to the connectedness lemma \cite[Theorem 5.48]{KM98},  we further see that $a_i<1$ for all $E_i$. We then prove the curve case.

If dim $Y>1$, fix a point $y\in Y$ and choose a general smooth divisor $Y_0$ passing $y$ so that $f|_{(X_0, D_0)}$ is still a semi-stable family with general fibers klt pairs, where we write $X_0=f^*Y_0$ and $D_0=D|_{X_0}$. Hence we have that $(X_0, D_0)$ is klt, by the induction hypothesis. By inversion of adjunction, we have $(X,D)$ is klt. 

The log canonical case can be proved similarly. 
\end{proof}


\subsection*{Fiber products of family of pairs}
Let us first introduce some notations of fiber products. For a morphism $f\colon X\rarr Y$ of schemes, define by
\[X^r_Y:=X\times_YX\times_Y\cdots\times_YX\]
the $r$-th fiber product and $f^r_Y\colon X^r_Y\rarr Y$ the induced morphism. If $\Gamma$ is a $\Q$-Cartier divisor on $X$, we write
\[\Gamma^r_Y=\sum_{i=1}^rp_i^*\Gamma\]
where $p_i\colon X^r_Y\rarr X$ the $i$-th projection; $X^r_Y$ and $\Gamma^r_Y$ are also denoted by $X^r$ and $\Gamma^r$ respectively if $Y$ is obvious from the context. 

\begin{lemma}\label{lm:prodsing}
If $(Z_1, D_1)$ and $(Z_2, D_2)$ are (stable) slc pairs, then $(Z_1\times Z_2, p_1^*D_1+p_2^*D_2)$ is also a (stable) slc pair. Moreover, if both of the pairs are log canonical (resp. klt), then so is $(Z_1\times Z_2, p_1^*D_1+p_2^*D_2)$.
\end{lemma}
\begin{proof}
This is the pair version of \cite[Theorem 3.2]{Ops03}. The proof works for the pair case as well, and we leave details for interested readers.
\end{proof}
Lemma \ref{lm:prodsing} together with Proposition \ref{prop:singtotal} immediately implies the following.
\begin{coro}\label{cor:productst}
If $f\colon (X,D)\rarr Y$ is a semi-stable family of slc pairs with general fibers log canonical (resp. klt.) pairs over a smooth variety $Y$, then the pair $(X^r_Y, D^r_Y)$ is log canonical (resp. klt) for every $r>0$.
\end{coro}

\subsection*{Singularities of divisors with respect to log smooth pairs}
For a line bundle $\cL$ on a smooth projective variety $X$ with $H^0(X, \cL)\neq 0$ and a normal crossing $\Q$-divisor $D$ with coefficients in $[0,1)$, we define 
\[e_D(\cL)\coloneqq \sup\{\frac{1}{\textup{lct}_D(B)}|B\in|\cL|\},\]
where lct$_D(B)$ is the log canonical threshold of $B$ with respect to the pair $(X,D).$  For the definition of $e_D(\cL)$ and results involving it below, we follow ideas in \cite[\S5]{Vie95}, where the case for $D=0$ has been discussed.
\begin{lemma}\label{lm:boundlc}
With $(X,D)$ and $\cL$ as above, if $\cA$ is a very ample line bundle, then we have for every integer $k>0$
\[e_{D}(\cL)\le C\cdot c_1(\cA)^{\textup{dim}X-1}\cdot c_1(\cL)
+1,\]
where the constant $C>0$ depends only on the coefficients of $D$.
\end{lemma}
\begin{proof}
We first set $a$ to be the maximum of the coefficeints of $D$ and $C=1/(1-a)$. We use induction on dimension to prove the statement.
If $X$ is a curve, then clearly 
\[e_D(\cL)\le C\cdot\textup{deg}(\cL)+1.\]
If dim$X>1$, for $\Gamma\in |\cL|$, consider a general section $H$ of $|\cA|$ so that $H+D$ is normal crossing and $H$ and $\Gamma$ do not share commnon components. The induction hypothesis tells us that 
\[e_{D|_H}(\cL|_H)\le C\cdot c_1(\cA|_H)^{\textup{dim}X-2}\cdot c_1(\cL|_H)+1=C\cdot c_1(\cA)^{\textup{dim}X-1}\cdot c_1(\cL)+1.\]
Then whenever $m>C\cdot c_1(\cA)^{\textup{dim}X-1}\cdot c_1(\cL)+1$, the multiplier ideal 
\[\cJ(H, 1/m\Gamma|_H+D|_H)\]
is trivial. By inversion of adjunction for multiplier ideals, $(X,1/m\Gamma+D)$ is klt.  Therefore, we have
\[e_{D}(\cL)\le C\cdot c_1(\cA)^{\textup{dim}X-1}\cdot c_1(\cL)
+1.\]
\end{proof}

We state a pair version of \cite[Proposition 5.19]{Vie95}, for later using. 

\begin{lemma}
Given a family $f:(X,D)\to Y$, with $(X,D)$ klt, all fibers smooth, $Y$ smooth, and an effective Cartier divisor $\Gamma$ on $X$. Fix a closed point $y\in Y$. Let $D'\subseteq D$ and $\Gamma'\subseteq \Gamma$ be those components that do not contain $X_y$. For any rational number $$m\ge e_{D'|_{X_y}}(\mathcal{O}_{X_y}(\Gamma'|_{X_y})),$$
assume that $X_y$ is not contained in the non-klt locus of $(X,D+1/m\Gamma)$, then there is an open neighborhood $U$ of $X_y$, such that $(U, D|_U+1/m\Gamma|_U)$ is klt.
\end{lemma}
\begin{proof}
For any rational number $m$ satisfies the condition, denote $\Delta_m=k(mD+\Gamma)$, for the smallest positive integer $k$ that makes $kD$ integral. It is not hard to check that $mk\ge e(\mathcal{O}_{X_y}(\Delta_m'|_{X_y}))$, where $\Delta'\subset \Delta$ are those components not containing $X_y$. Applying  \cite[Proposition 5.19]{Vie95} on $\Delta_m$, we have that, if $X_y$ is not contained in the non-klt locus of $(X,1/km\Delta_m)=(X,D+1/m\Gamma)$, then there is an open neighborhood $U$ of $X_y$, such that $(U, 1/km\Delta_m|_U)=(U, D|_U+1/m\Gamma|_U)$ is klt. 
\end{proof}

\begin{lemma}\label{lm:prode}
Considering the $r$-th fiber product $(X^r, D^r)$ of the pair $(X,D)$ and $$\cB=\bigotimes^r_i p_i^*\cL,$$ 
in the situation of Lemma \ref{lm:boundlc}, we have $e_{D^r}(\cB)=e_D(\cL).$
\end{lemma}
\begin{proof}

By Lemma \ref{lm:prodsing}, one has $e_{D^r}(\cB)\ge e\coloneqq e_D(\cL)$. To prove $e_{D^r}(\cB)\le e$, we use induction on $r$. Consider $pr_i:(X^r, D^r)\to X$, the natural projection onto the $i$-th factor of $X^r$. 
Fix any geometric fiber $X^r_x$ of $pr_i$, and any section $\Gamma$ of $\cB$, with the same notations as in the previous lemma, we have $e\ge e_{D^{r-1}}(\cO_{X^r_x}(\Gamma'|_{X^r_x}))$. 
This is because the fact that $\cO_{X^r_x}(\Gamma'|_{X^r_x})=\cO_{X^r}(\Gamma')|_{X^r_x}\subset \cB|_{X^r_x}$, and the induction hypothesis. This implies that $e$ satisfies the requirement in the previous lemma. Hence we have the non-klt locus of $(X^r, D^r+1/m\Gamma)$ is of the form $pr_i^{-1}(T_i)$, for some subscheme $T_i$ of $X$. Since this holds for all $i$, the non-klt locus has to be empty.
\end{proof}


\begin{coro}\label{cor:uniconst}
Suppose that $f: (U,D)\to V$ is a log smooth family with $(U,D)$ a klt pair. Then there exists a constant $C>0$ depending on the dimension of fibers and the coefficients of $D$, so that for every integer $r>0$ and every $y\in V$, 
\[e_{D^r}(\omega^k_{U^r_y}(kD_y^r))<Ck\]
for all $k>0$ with $kD$ integral, where $(U_y, D_y)$ is the fiber over $y$.
\end{coro}
\begin{proof}
For a fixed $y\in V$, by Lemma \ref{lm:boundlc} and Lemma \ref{lm:prode}, there exists a constant $C$ depending on the dimension of the fiber $X_y$ and the coefficients of $D_y$ so that 
\[e_{D^r}(\omega^k_{U^r_y}(kD_y^r))<Ck\]
for all $r$ and all $k$ with $kD$ integral. We then can choose an $C$ working for all $y\in V$, by using the fact that $e_{D^r}(\omega^k_{U_y}(kD_y))$ is upper semicontinuous as a function of $y\in V$, which will be proved later. 

Now we prove the semicontinuity of $e_{D^r}(\omega^k_{U^r_y}(kD^r_y))$. By invariance of plurigenera in the log case \cite[Theorem 4.2]{HMX18}, $f^r_*\omega^k_{U^r/V}(kD^r)$ is locally free for all $k>0$ with $kD$ integral. Set $\mathbb P=\mathbb P(f^r_*\omega^k_{U^r/V}(kD^r))$. We consider the universal divisor $B$ on $\mathbb P\times_V (U^r,D^r)$ parametrizing the linear system $|\omega^k_{U^r_y}(kD_y^r)|$ for all $y\in V$. It is well-known that lct$_{\mathbb P_x\times_V D^r}(B_x)$ is lower semicontinuous as a function of $x\in \mathbb P$. One then easily get the upper semicontinuity of $e_{D^r}(\omega^k_{U_y}(kD_y))$ from the  lower semicontinuity of lct$_{\mathbb P_x\times_V D^r}(B_x)$.
\end{proof}

\subsection*{Stable reduction of log smooth families}
Suppose that $f_V\colon (U, D_U)\to V$ is a log smooth family of projective klt pairs of log general type and $V$ smooth quasi-projective. Using \cite{BCHM}, we consider the relative log canonical model of $f_V$
\[f_V^c\colon (U_c, D_c)\rarr V),\]
which gives a stable family (by deformation invariance of plurigenera in the log case), birational to the original family $f$, where $D_c$ is the strict transform of $D_U$. 

Since $f_V^c$ is a stable family of klt pairs, we can compactify the base and get a stable family over a projective base, by using \cite[Corollary 6.18]{KP16}. More precisely, we have the following commutative diagram:
\begin{center}
    \begin{tikzcd}
    U\arrow[hook]{r} \arrow{d}&
    W\arrow{d}{f}&
    Z\arrow[l,dashrightarrow]\arrow{d}{h}\\
    V\arrow[hook]{r}&
    Y&
    Y'\arrow{l}{\tau},
    \end{tikzcd}
\end{center}
where $V\hookrightarrow Y$ is a compactification with $Y$ projective smooth and $E:= Y\setminus V$ is normal crossing so that there exists a generic finite morphism $\tau\coloneqq Y'\to Y$ with $Y'$ normal and a stable family of slc pairs $h\colon (Z, D)\to Y'$ satisfying 
$$h|_{V'=\tau^{-1}V}=f^c_{V'},$$ 
where $f^c_{V'}$is the induced family of $f^c_V$ after base change $\tau|_{V'}\colon V'\to V$. 

The following reduction process will help us make $\tau$ finite and $Y'$ smooth. By the Raynaud-Gruson flattening theorem (see \cite[Th\'eor\`em 5.2.2]{RG71}), there exists a birational map $\overline{Y}\to Y$ so that $Y'\times_Y\overline{Y}\to\overline{Y}$, the base change of $\tau$, is flat and hence also finite. Replacing $Y$ by $\overline{Y}$ and $Y'$ by the main component of $Y'\times_Y\overline{Y}$, we can hence assume $\tau$ is finite. 
Replacing $Y$ by a further resolution, we can then assume $\tau$ is finite and branched along a normal crossing divisor $\Delta(Y'/Y)$ containing $E$. Then, using Kawamata's covering (see for instance \cite[Corollary 2.6]{Vie95}), there exist a projective manifold $Y_1'$ and a finite map $Y_1'\to Y'$. After replacing $Y'$ by $Y_1'$ and $\tau$ by the composition of the two finite morhism, we can further assume that $Y'$ is a projective manifold and $\tau$ is finite. We also replace the stable family $f'$ by the induced family over the new $Y'$.

Set $D_W$ to be the closure of $D_U$ in $W$. After taking further resolution of the pair $(W,D_W)$ with centers away from $U$, we can assume that 
\[\omega^m_{W/Y}(mD_W)|_U=\omega^m_{U/V}(mD_U)\] 
for all $m>0$ sufficiently divisible and $f^{-1}\Delta(Y'/Y)+\supp(D_W)$ is normal crossing (hence so is $f^{-1}E$). Then we take $W'$ to be a resolution of the main component of $W'\times_Y Y'$ so that 
\[W'|_{g^{-1}V'}=U\times_VV'.\]
In summary, we have a commutative diagram 
\begin{equation}\label{eq:diag1}
\begin{tikzcd}
  (U, D_U) \arrow[r, hook] \arrow[d, "f_V"] & (W, D_W) \arrow[d, "f"] & W'\arrow[l] \arrow[d, "g"]\\
  V \arrow[r, hook] &Y & Y'.\arrow[l, "\tau"]
  \end{tikzcd}
\end{equation}
Since fibers of $f_V$ are of log general type, we know that $f_V$ and $f_V^c$ are birational and hence so are $W'$ and $Z$.

\begin{prop}\label{prop:stablered}
With the notations as above, after replacing the original family $f_V$ by a birational model, there exists a commutative diagram 
\begin{equation}\label{eq:reddiagram}
\begin{tikzcd}
  (U, D_U) \arrow[r, hook] \arrow[d, "f_V"] & (W, D_W) \arrow[d, "f"] & (W', D')\arrow[l, "\tau_W"] \arrow[d, "g"]\arrow[r,"\sigma"] & (Z,D)\arrow[d, "h"]\\
  V \arrow[r, hook] &Y & Y'\arrow[l, "\tau"] \arrow[r, "="]& Y'
\end{tikzcd}
\end{equation}
satisfying 
\begin{enumerate}[label=(\alph*)]
\item $h\colon (Z,D)\to Y'$ is a stable family of slc pairs and $(Z, D)$ is klt;\label{item:streda}
\item $\sigma$ is birational, $W'$ is smooth, $(W', D')$ is klt and  the support of $D'$ is normal crossing;\label{item:stredb}
\item for all $m>0$ sufficiently large and divisible, there exists an isormorphism
\[g_*\cO_{W'}(m(K_{W'/Y'}+ D'))\stackrel{\simeq}{\rarr} h_*\cO_Z(m(K_{Z/Y'}+D)),\]
and they are both reflexive sheaves over $Y'$.\label{item:stredc}
\item for all $m$ sufficiently large and divisible, there exists an injective morphism
\[g_*\omega^m_{W'/Y'}(m D')\rarr \tau^*f_*\omega^m_{W/Y}(m(D_W+f^*E))\]
which is an isomorphism over $V'=\tau^{-1}(V)$, \label{item:stredd}
\end{enumerate}
\end{prop}
\begin{proof}
First, by construction, we know that $h\colon (Z,D)\to Y'$ is stable, general fibers are klt pairs (fibers over close points in $V$) and special fibers are slc pairs. Therefore, thanks to Proposition \ref{prop:singtotal},  $(Z, D)$ is also a klt pair. Then \ref{item:streda} follows.

Since $W'$ in Diagram \eqref{eq:diag1} is birational to $Z$, we replace $W'$ by a common resolution of $W'$ and $(Z, D)$. We write   
\[K_{W'}+D_{W'}=\sigma^*(K_{Z}+D),\]
and $D_{W'}={D_{W'}}^+-{D_{W}'}^-$ with $ {D_{W'}}^+$ and ${D_{W'}}^-$ effective. Clearly ${D_{W'}}^-$ is exceptional. Then we set
\[ D':=D_{W'}-\floor{-{D_{W'}}^-},\]
and \ref{item:stredb} follows.

Since ${D_{W'}}^-$ is exceptional, 
by projection formula, we know 
\[\sigma_*\cO_{W'}(m(K_{W'/Y'}+D'))\simeq \cO_{Z}(m(K_{Z/Y'}+D))\]
for all sufficiently divisible $m>0$. Pushing forward the above isomorphism by $h$, we obtain the isomorphism in \ref{item:stredc}. Since $h$ is flat, we know that $h_*\cO_Z(m(K_{Z/Y'}+D))$ is reflexive and hence \ref{item:stredc} also follows. 

To prove \ref{item:stredd}, we first apply \cite[Lem. 3.2]{Vie83} and obtain for every Cartier divisor $A$ on $W$ and $k\ge 0$
\begin{equation}\label{eq:morph1}
g_*\omega^{k+1}_{W'/Y'}(\tau_W^*A)\to  \tau^*f_*\omega^{k+1}_{W/Y}(A),
\end{equation}
which is isomphic over $V'$. We also need an injective morphism 
 \begin{equation}\label{eq:morph2}
g_*(\omega^m_{W'/Y'}(mD')) \hookrightarrow g_*(\omega^m_{W'/Y'}(m\tau_W^*(D_W+f^*E)).
 \end{equation} 
Now we started to construct it. By construction, we know $(h^{-1}(V'),D|_{h^{-1}(V')})$ is the relative log canonical model of the log smooth family 
$$(U, D_U)\times_VV'\rarr V',$$ 
and hence 
we know that for all sufficiently divisible $m>0$
\begin{equation}\label{E:iso over V'}
    g_*(\omega^m_{W'/Y'}(m\tau_W^*(D_W))|_{V'}\simeq g_*(\omega^m_{W'/Y'}(mD'))|_{V'}.
\end{equation}
For a open subset $U'\subseteq Y'$ and a section 
$$s\in \Gamma(g_*(\omega^m_{W'/Y'}(m D')),U')=\Gamma(\omega^m_{W'/Y'}(m D'),g^{-1}U'),$$
we know $(s)+(K_{W'/Y'}+m D')|_{g^{-1}U'}\ge 0$; here we treat $s$ as a rational function and $(s)$ is the divisor of $s$. Since the coefficients of $D'$ are in $(0,1)$, we know that for every prime divisor $C$, if $C$ is supported on the support of $D'$ and the support $\tau_W^*(D_W+f^*E)$, then 
\[\textup{ord}_C(D')\le \textup{ord}_C(\tau_W^*(D_W+f^*E)).\]
Hence, if 
\[s\notin \Gamma(\omega^m_{W'/Y'}(\tau_W^*(D_W+f^*E),g^{-1}U'),\]
then $s$ must have poles along prime divisors supported on the support of $ D'$ but not on the support $\tau_W^*(D_W+f^*E)$. Hence, we obtain
\[s|_{g^{-1}(\tau^{-1}V)}\notin \Gamma(\omega^m_{W'/Y'}(\tau_W^*(D_W+f^*E),g^{-1}(\tau^{-1}V)).\] 
However, since $g_*(\omega^m_{W'/Y'}(m\tau_W^*D_W))|_{\tau^{-1}V}\simeq g_*(\omega^m_{W'/Y'}(m D'))|_{\tau^{-1}V}$, this can not happen.  
Therefore, we obtain the morphism \eqref{eq:morph2}.  Clearly it is isomorphic over $\tau^{-1}{V}$.

Taking $k=m-1$ and 
$$A=m(D_W+f^*E)$$ in \eqref{eq:morph1}, we get a morphism 
\[g_*\omega^{m}_{W'/Y'}(m\tau_W^*(D_W+f^*E))\to  \tau^*f_*\omega^{m}_{W/Y}(m(D_W+f^*E)).\] 
Composing it with \eqref{eq:morph2}, we obtain the desired morphism in \ref{item:stredd},
\[g_*\omega^m_{W'/Y'}(mD')\rarr \tau^*f_*\omega^m_{W/Y}(m(D_W+f^*E)).\]
\end{proof}

\section{Generic freeness of direct images of pluri-canonical sheaves}\label{sect:genfree}
In this section, we will prove generic freeness of direct images of pluri-canonical sheaves, based on the stable reduction constructed in \S \ref{sect:stred}. This property will be used in \S \ref{sect:higgs}, for constructing the Higgs sheaf.

The strategy we use is mainly the klt-pair generalization of  that in \cite[Appendix]{PTW18}. 
We also need an analytic extension theorem of Cao \cite[Theorem 2.10]{Cao}, which is essentially due to Berndtsson, P\u{a}un and Takayama. We present here a klt-pair version for the application in this article, which is an immediate consequence of the proof in $loc.$ $cit.$

\begin{theorem}\label{thm:Cao}
Let $f:(X,D)\to Y$ be a family between smooth projective varieties with $(X,D)$ being a klt pair and let $\cL$ be a line bundle on $X$ with a singular metric $h$ such that $i\Theta_{h}(\cL)\ge 0$ in the sense of currents. Let $\cA$ be a very ample line bundle on $Y$ such that the global sections of $\cA\otimes \omega_Y^{-1}$ separate $2n$-jets,  where $n$ is the dimension of $Y$, and let $V\subseteq Y$ be a Zariski open set such that $f$ is flat over $V$ and $h^0(X_y, \omega^k_{X/Y}(kD)\otimes \cL|_{X_y})$ is constant over $y\in V$, for an integer $k$ with $kD$ integral.  Assume also  that the analytic multiplier ideal 
$$\mathcal J((h+h_{kD})^{\frac{1}{k}}|_{X_y})=\sO_{X_y},$$ 
for $y\in V,$ where $h_D$ is the natural singular metric given by $D$. 
Then,
\[f_*(\omega^k_{X/Y}(kD)\otimes \cL)\otimes \cA\]
is globally generated over $V$.
\end{theorem}
It is worth mentioning that the choice of $\cA$ depends only on $Y$ but not on the family $f$; see \cite[Remark 2.11]{Cao}.

 \begin{theorem}\label{thm:ggrefine}
Let $f_V\colon (U, D)\to V$ be a log smooth family of log general type klt pairs with maximal variation. After replacing the family $f_V$ by a birational model, there exists a projective family of pairs $f_X:(X, D_X)\to Y$ with $X$ and $Y$ smooth projective, the coefficients of $D$ in $(0,1)\cap \Q$, $E=Y\smallsetminus V$ a simple normal crossings divisor, and $f_X^{*}(E)+D_X$ normal crossing, as well as an ample line bundle $\cA$ on $Y$, such that 
\begin{enumerate}
\item $f_X|_{f^{-1}V}$ is the fiber-product family $f^r_V: (U^r, {D^U}^r)\to V$ for some $r>0$;
\item the sheaf 
$$f_{X*}\omega^k_{X/Y}(kD_X) \otimes \cA(E)^{-k}\big)$$
is generically globally generated for all sufficient large and divisible $k$.  
\end{enumerate}
\end{theorem}
\begin{proof}
First, we do stable reduction of $f_V$ as in Proposition \ref{prop:stablered}. Then replacing $f_V$ by a birational model, we obtain a Diagram 
\[\begin{tikzcd}
  (U, D_U) \arrow[r, hook] \arrow[d, "f_V"] & (W, D_W) \arrow[d, "f"] & (W', D')\arrow[l, "\tau_W"] \arrow[d, "g"]\arrow[r,"\sigma"] & (Z,D)\arrow[d, "h"]\\
  V \arrow[r, hook] &Y & Y'\arrow[l, "\tau"] \arrow[r, "="]& Y'
\end{tikzcd}\]

We then fix an ample line bundle $\cA$ on $Y$ so that $\cA(E)$ is also ample and the global sections of $\tau^*\cA\otimes \omega_{Y'}^{-1}$ seperate all $2n$-jets, where $E=Y\setminus V$. 

Since $f_V$ is of maximal variation, by definition so is $h$. By \cite[Theorem 7.1]{KP16}, we know $h_{*} \omega^{m} _{Z/Y'}(mD))$ is big and hence so is $[\det] h_{*} \omega^{m} _{Z/Y'}(mD)$ for all $m$ sufficiently large and divisible, as the determinant of a big sheaf is also big. Here and in what follows, we use $[\bullet]$ to denote the reflexive hull of the algebraic operator $\bullet$. We then fix an $m$ sufficiently large and divisible. By bigness we can assume
\[\big([\det] h_{*} \omega^{m} _{Z/Y'}(mD)\big)^v=\cA'^{2(C+1)m}(E'')\]
for some $v>0$ (depending on $m$ and $C$) and some effective divisor $E''$ on $Y'$, where $\cA'=\tau^*\cA(E)$ and $C$ is the uniform constant (assumed to be an integer $>2$) as in Corollary \ref{cor:uniconst} for the family $g\colon (W',D')\to Y'$ over its log smooth locus $V''$.

We now take $r=vr_0$, where $r_0=\textup{rank}(h_{*} \omega^{m} _{Z/Y'}(mD))$. 
We apply Proposition \ref{prop:stablered} again and obtain the reduction diagram for $f^r_V$
\[\begin{tikzcd}
  (U^r, D^r_U) \arrow[r, hook] \arrow[d, "f^r_V"] & (X, D_X) \arrow[d, "f_X"] & (X', D'_{X})\arrow[l, "\tau_X"] \arrow[d, "g'"]\arrow[r,"\sigma"] & (Z^r,D^r)\arrow[d, "h^r"]\\
  V \arrow[r, hook] &Y & Y'\arrow[l, "\tau"] \arrow[r, "="]& Y'
\end{tikzcd}\]
By construction, we know that the reduction diagram of $f^r_V$ is generically taking $r$-th fiber product of the reduction diagram of $f_V$. Hence, we can assume 
\begin{equation}\label{eq:fprodg}
 g'|_{g'^{-1}(V'\cap V'')}=(g|_{(g)^{-1}(V'\cap V'')})^r,
\end{equation}
where $V'=\tau^{-1}(V)$ and $V''$ is the log smooth locus of $g$.

By Proposition \ref{prop:stablered} (c), for all $l$ sufficiently large and divisible
\[h_{*}\omega^l_{Z/Y'}(lD)\simeq g_*\omega^l_{W'/Y'}(lD')\]
and 
\[h^r_{*}\omega^l_{Z^r/Y'}(lD^r)\simeq g'_*\omega^l_{X'/Y'}(lD'_X),\]
and they are all reflexive sheaves. By flat base change and reflexivity, we know that 
\[g'_*\omega^l_{X'/Y'}(lD'_X)\simeq h^r_{*}\omega^l_{Z^r/Y'}(lD^r)\simeq [\bigotimes^r] h_{*} \omega^l _{Z/Y'}(l D).\]
On the other hand, we have the natural morphism 
\[[\det] h_{*} \omega^{m} _{Z/Y'}(mD) \longrightarrow [\bigotimes^{r_0}] h_{*} \omega^{m} _{Z/Y'}(mD),\] 
splitting locally over $V'$. Hence, we obtain morphisms
\[\cA'^{2(C+1)m}\rarr \big([\det] h_{*} \omega^{m} _{Z/Y'}(mD)\big)^v\rarr g'_{*} \omega^{m} _{X'/Y'}(mD'_{X}),\]
where the composition splits locally over $V'\setminus E''$.  The composition above gives an effective divisor 
$$\Gamma\in |\omega^{m}_{X'/Y'}(mD'_X)\otimes g'^*\cA'^{-2(C+1)m}|$$ 
and it does not contain any fibers over $(V')\setminus E''$. By Corollary \ref{cor:uniconst} and \eqref{eq:fprodg}, we know that
\[(X_y,(\frac{1}{Cm}\Gamma+D')|_{X'_y})\]
is klt for every $y\in V'\cap V''$. For every $l>0$, we apply Theorem \ref{thm:Cao} to the family $g'$ and the line bundle 
\[\omega^{lm}_{X'/Y'}(lmD'_X)\otimes g'^*\cA'^{-2(C+1)lm}\]
with the natural singular metric induced by $l\Gamma$. Since $(X_y,(\frac{1}{Cm}\Gamma+D')|_{X'_y})$ is klt, we conclude that 
\[g'_*\omega^{(C+1)lm}_{X'/Y'}((C+1)lmD'_{X})\otimes\cA'^{-2(C+1)lm}\otimes \tau^*A\]
is generically globally generated. 

By Proposition \ref{prop:stablered} (d), we have a morphism 
\[g'_*\omega^{(C+1)lm}_{X'/Y'}((C+1)lm D'_X)\rarr \tau^*f_{X*}\omega^{(C+1)lm}_{X/Y}((C+1)lm(D_X+f_X^*E))\]
which is isormophic generically. Hence, the sheaf
\[\tau^*f_{X*}\omega^{(C+1)lm}_{X/Y}((C+1)lm(D_X+f_X^*E))\otimes\cA'^{-2(C+1)lm}\otimes \tau^*\cA\]
is also generically globally generated. 
Since $\tau$ is finite, we are allowed to apply projection formula for coherent sheaves to obtain a generically surjective morphism 
\[\bigoplus\tau_*\cO_{Y'}\rarr f_{X*}\omega^{(C+1)lm}_{X/Y}((C+1)lmD_X)\otimes\cA^{-2(C+1)lm+1}(-(C+1)lmE)\otimes\tau_*\cO_{Y'}.\]
Since $\cA$ is ample by assumption, we pick $l$ sufficiently large so that 
$$\tau_*\cO_{Y'}\otimes \cA^{(C+1)lm-1}$$ 
is globally generated. Then using the trace map
\[\tau_*\cO_{Y'}\rarr \cO_Y,\]
we conclude that 
\[f_{X*}\omega^{k}_{X/Y}(kD_X)\otimes\cA(E)^{-k}\]
is generically globally generated, where we take $k=(C+1)lm$.
\end{proof}

One can observe that the proof of the above theorem shows more precisely that $f_{X*}\omega^{k}_{X/Y}(kD_X)\otimes\cA(E)^{-k}$ is generically generated by sections belonging to the subspace $\mathbb V_k$ defined as the image of the map 
\[H^0(Y',h^r_{*}\omega^{k}_{Z^r/Y'}(k{D^r})\otimes\tau^*(\cA(2E)^{-k}))\to H^0(Y,f_{X*}\omega^{k}_{X/Y}(kD_X)\otimes\cA(E)^{-k}), \]
induced by the injective morphism in Proposition \ref{prop:stablered} (d). Then one applies the arguments in the proof of \cite[Proposition 4.4]{PTW18} and obtains the following proposition regarding lifting sections in $\mathbb V_m$ through a birational base change. It will be needed in a reduction step for the construction of Higgs sheaves.  

\begin{prop}\label{prop:liftsections}
In the situation of Theorem \ref{thm:ggrefine}, let $S$ be effective divisor with $S\ge \Delta_\tau$, the discriminant divisor of $\tau$, and let $\psi\colon \tilde Y\to Y$ be a log resolution of $E+S$ with centers in the singular locus of  $E+S$. Then for every $s\in \mathbb V_k$, there exist a subset $T\in \tilde Y$ of codimension at least 2, a birational model $\tilde f\colon (\tilde X, D_{\tilde X})\to \tilde Y$ of $f$ with $(\tilde X, D_{\tilde X}+\tilde f^*\psi^*(E))$ log smooth and the coefficients of $D_{\tilde X}$ is in $[0,1)$, and a section 
\[\tilde s\in H^0(\tilde Y_0,\tilde f_*\omega^{k}_{\tilde X/\tilde Y}(kD_{\tilde X})\otimes\psi^*\cA(E)^{-k}) \]
with $\tilde Y_0=\tilde Y\setminus T$ so that 
\[\tilde s|_{\psi^{-1}(V\setminus S)}=\psi^*(s|_{V\setminus S}).\]
\end{prop}

\section{Hodge theoretical constructions of Higgs sheaves}\label{sect:higgs}

\subsection*{Construction of Higgs sheaves for boundary $\Q$-divisors}
In this section, we give a construction of Higgs sheaves for families of pairs. We mainly follow the ideas in \cite[\S 2.3 and 2.4]{PS17}. 

We assume that $f\colon X\to Y$ is a surjective projective morphism between smooth projective varieties with connected general fibers branched along a normal crossing divisor $D_f$. We also assume that  $D$ is an effective normal crossing divisor with coefficients in $(0,1)$ on $X$ and $\cA$ is a line bundle on $Y$. We make the following assumption
\begin{equation}\label{as:sections}
0\neq H^0(X, \omega^m_{X/Y}(mD)\otimes f^*\cA^{-m}),
\end{equation}
for some $m$ sufficiently divisible. We set $\cL=\omega_{X/Y}(\ceil D)\otimes f^*\cA^{-1}$. Then $\cL^m$ has a section $s$ with its zero divisor $(s)$ containing the support of $D$.

We take the $m$-th cyclic covering of the section $s$, $\pi\colon X_m\to X$. Let $\nu\colon Z \to X_m$ be a desinularization of the normalization of $X_m$. Set $\phi=\pi\circ\nu$ and $h=f\circ\phi$. By the construction of cyclic coverings, $\phi^*\frac{1}{m}(s)$ is an effective divisor on $Z$ that containing $D_Z$, the support of 
$\phi^*(D)$. Replacing $Z$ by a further resolution, we can assume that $D_Z$ is normal crossing. We then have an induced inclusion 
\[
\phi^*\cL^{-1}\hookrightarrow \cO_Z(-D_Z).
\]
After composing with $\phi^*\Omega^{k}_X(\log D)\to \Omega^k_Z(\log D_Z)$, we hence obtain injective morphisms 
\begin{equation}\label{eq:tautoeq}\phi^*\big(\cL^{-1}\otimes\Omega^{k}_X(\log D)\big)\hookrightarrow\Omega^k_Z(\log D_Z)(-D_Z)\hookrightarrow \Omega^k_Z
\end{equation}
for $k=0,1,\dots,n$.

Consider the log Spencer complex relative over $Y$ 
\[C_{(X,D),\bullet}\coloneqq[f^*\sA_Y^{\bullet-n}\otimes\bigwedge^n\shTA_X(-\log D)\to f^*\sA_Y^{\bullet-n+1}\otimes\bigwedge^{n-1}\shTA_X(-\log D)\to\cdots \to f^*\sA_Y^{\bullet}],\]
starting from the $-n$ term and similarly \[C_{Z,\bullet}\coloneqq[f^*\sA_Y^{\bullet-n}\otimes\bigwedge^n\shTA_Z\to f^*\sA_Y^{\bullet-n+1}\otimes\bigwedge^{n-1}\shTA_Z\to\cdots \to f^*\sA_Y^{\bullet}].\]

\begin{lemma}
The inclusion \eqref{eq:tautoeq} induces a none-trivial morphism of complexes of graded $\sA_Y^\bullet$-modules
\[{\bf R}f_*\big( \cL^{-1}\otimes\omega_{X/Y}(\ceil D)\otimes C_{(X,D),\bullet}\big)\to{\bf R}h_*\big(\omega_{Z/Y}\otimes C_{Z, \bullet}\big). \]
\end{lemma}
\begin{proof}
By adjunction, it is enough to construct a morphism 
\[\phi^*\big(\cL^{-1}\otimes\omega_{X/Y}(\ceil D)\otimes C_{(X,D),\bullet}\big)\to \omega_{Z/Y}\otimes C_{Z, \bullet}.\]
Using the dual pair $\Omega_X(\log D)$ and $\shTA_X(-\log D)$, we have a natural isomorphism 
\[\cL^{-1}\otimes\omega_{X/Y}(\ceil D)\otimes C_{X,-k}\simeq \cL^{-1}\otimes f^*\omega^{-1}_{Y}\otimes \Omega^{n-k}_X(\log D),\]
and hence a natural isomorphism 
\[\phi^*\big(\cL^{-1}\otimes\omega_{X/Y}(\ceil D)\otimes C_{X,-k}\big)\simeq \phi^*\cL^{-1}\otimes h^*\omega^{-1}_{Y}\otimes \phi^*\Omega^{n-k}_X(\log D).\]
Using the composition in \eqref{eq:tautoeq}, we thus obtain morphisms 
\[\phi^*\big(\cL^{-1}\otimes\omega_{X/Y}(\ceil D)\otimes C_{(X,D),-k}\big)\to \omega_{Z/Y}\otimes C_{Z, -k}\]
for $k=1,2,\dots,n$. Checking the individual morphisms compatible to the differentials of the complexes are strait-forward and skipped. 
\end{proof}

Applying \cite[Proposition 2.4]{PS17}, we have 
\[\gr^F_\bullet \cH^0h_+\Q^H_Z[n]\simeq {\bf R}^0h_*\big(\omega_{Z/Y}\otimes C_{Z, \bullet}\big).\]
We set $\cM$ to be the Hodge module with the support $X$ in the strict-support decomposition of $\cH^0h_+\Q^H_Z[n]$.

We then define $\cG_\bullet$ to be the graded $\sA_Y^\bullet$-module as the image of the composition
\[{\bf R}^0f_*\big( \cL^{-1}\otimes\omega_{X/Y}(\ceil D)\otimes C_{(X,D),\bullet}\big)\rarr{\bf R}^0h_*\big(\omega_{Z/Y}\otimes C_{Z, \bullet}\big)\rarr \gr^F_\bullet\cM.\]

For simplicity, we set $\widetilde\cG^i_\bullet={\bf R}^if_*\big( \cL^{-1}\otimes\omega_{X/Y}(\ceil D)\otimes C_{(X,D),\bullet}\big)$ for all integers $i$. By definition, $\widetilde\cG^i_\bullet$ is not necessarily a coherent $\cO_Y$-module, if the grading structure is forgotten. However, it is coherent over the locus where the family of pairs is log smooth as showing in the following lemma. It is a replacement of \cite[Proposition 2.10]{PS17} in the pair case, which will be used in proving Proposition \ref{prop:constructionHiggs}. 
\begin{lemma}\label{lm:coherence}
With notations as above, we have 
\[\widetilde\cG^i_m|_{U_s}=0\]
for every $i$ and every $m\gg0$, where $U_s \subset Y$ the Zariski open subset over which the morhphism $f: (X, D)\to Y$ is log smooth. In particular, 
\[\cG_m|_{U_s}=0\]
for every $m\gg0$.
\end{lemma}
\begin{proof}
Applying \cite[Proposition 3.1]{MP16}(see also \cite{W17a} and \cite{Wu18}  for further discussion of more general results), we have a quasi-isomorphism
\[C_{(X,D),\bullet}\simeq \gr^F_\bullet\cO_X(*D)\stackrel{{\bf L}}{\otimes}_{\sA^\bullet_X}f^*\sA^\bullet_Y.\]
Here $(\cO_X(*D),F_\bullet)$ is the filtered $\sD_X$-module with $$\cO_X(*D)=\lim_{k\to \infty}\sO_X(k\ceil D)$$
and $F_p\cO_X(*D)=\sO_X((p+1)\ceil D)$ if $p\ge0$ and $0$ otherwise.

Then, we set $V_s=f^{-1}U_s$ and $d=\dim X-\dim Y$. Over $V_s$ (or more generally over the smooth locus of $f$), $f^*\sA_Y^\bullet$ has a locally free (finite) resolution (the relative Spencer resolution)
\[
[\sA_X^{\bullet-d}\otimes \stackrel{d}{\bigwedge}\shTA_{X/Y}\cdots\to\sA_X^{\bullet-i}\otimes \stackrel{i}{\bigwedge}\shTA_{X/Y}\to\cdots\to  \sA_X^{\bullet}]\to f^*\sA_Y^\bullet.
\]
where $\shTA_{X/Y}$ is the relative tagent sheaf.
To see this, one observes that the graded complexe is the Koszul complex of $\sA_X^{\bullet}$ with actions of multiplications by $\partial_{x_1},\dots,\partial_{x_{d}}$, when $\partial_{x_1},\dots,\partial_{x_{d}}$ are free generators of $\shTA_{X/Y}$ locally. 

Using the above resolution, we get a quasi-isomorphism 
\[\cL^{-1}\otimes\omega_{X/Y}(\ceil D)\otimes C_{(X,D),\bullet}\simeq f^*\cA\otimes[\gr^F_{\bullet-d}\cO_X(*D)\otimes \stackrel{d}{\bigwedge}\shTA_{X/Y}\to\cdots\to  \gr^F_{\bullet}\cO_X(*D)].\]
Since $D$ is relative normal crossing over $U_s$, using the argument in the proof of \cite[Proposition 2.3]{PTW18}, the complex 
\[\gr^F_{\bullet-d}\cO_X(*D)\otimes \stackrel{d}{\bigwedge}\shTA_{X/Y}\to\cdots\to  \gr^F_{\bullet}\cO_X(*D)\]
is quasi-isomorphic to the complex
\[\stackrel{d}{\bigwedge}\shTA_{X/Y}(-\log D)\to\cdots \to \shTA_{X/Y}(-\log D) \to  \cO_X\]
over $V_s$, where $\shTA_{X/Y}(-\log D)$ is the relative tagent sheaf with logarithmic zeros along D. Therefore $\cL^{-1}\otimes\omega_{X/Y}(\ceil D)\otimes C_{(X,D),\bullet}$ is quasi-isomorphic to 
\[f^*\cA\otimes\stackrel{d}{\bigwedge}\shTA_{X/Y}(-\log D)\to\cdots \to f^*\cA\otimes\shTA_{X/Y}(-\log D) \to f^*\cA\]
over $V_s$.
This means that $\widetilde\cG^i_\bullet|_{U_s}$ is a coherent $\cO_{U_s}$-module when the grading is forgotten. Therefore, we obtain
\[\widetilde\cG^i_m|_{U_s}=0\]
for every $m\gg0$.
\end{proof}

The following proposition regarding the construction of Higgs sheaves in codimension one for log smooth families is a natural generalization of \cite[Theorem 2.3]{PS17}. 
\begin{prop}\label{prop:constructionHiggs}
Let $f\colon X\rarr Y$ be a surjective projective morphism between smooth projective varieties with connected general fibers. Assume that $D$ is an effective normal crossing $\Q$-divisor with coefficients in $[0,1)$ on $X$ and $D_f$ is a reduced normal crossing divisor on $Y$ so that $(X, D)$ is log smooth over $Y\setminus D_f$ and $\cA$ a line bundle on $Y$. Assuming \eqref{as:sections}, one can find a graded $\sA^\bullet_Y(-\log D_f)$-module $\sF_\bullet$ satisfying the following properties:
\begin{enumerate}
\item $\sF_0$ is a line bundle, satisfying $\cA(-D_f)\subseteq \sF_{0}$.
\item $\sF_m$ is reflexive for $m>0$.
\item There exist a reduced divisor $E\supseteq D_f$, and a polarizable VHS defined over $Y\setminus E$ so that 
\[\sF_\bullet|_{Y\setminus \Lambda}\subseteq\sE_\bullet\]
as graded $\sA^\bullet_{Y\setminus \Lambda}(-\log E|_{Y\setminus \Lambda})$-modules, where $\sE_\bullet$ is the logarithmic Higgs bundle associated to the canonical extension of the VHS with eigenvalues in $[0,1)$ along $E|_{Y\setminus Z}$, where $\Lambda$ is the locus where $E$ is not normal crossing.
\end{enumerate}
\end{prop}
\begin{proof}
By the assumption \eqref{as:sections}, we obtain the Higgs sheaf $\cG_\bullet$ and the Hodge module $\cM$. By construction, we know $\cM$ is the minimal extension of a geometric VHS defined over an open dense subset of $Y$. After possibly shrinking the open set, we are allowed to assume its complement is a divisor. We then set $E$ to be the union of $D_f$ and the divisor and set $\Lambda$ to be the locuse where $E$ is not normal crossing ($\Lambda$ is of codimension at least 2). We also take $\sE_\bullet$ to be the logarithmic Higgs bundle associated to the canonical extension of the VHS with eigenvalues in $[0,1)$ along $E|_{Y\setminus \Lambda}$.

Using the arguments in \cite[Proposition 2.9]{PS17}, we have 
\begin{equation}\label{eq:inc111}
\cG_0=\cA\otimes f_*\cO_X\subseteq \gr^F_0 \cM.
\end{equation}
Therefore, $\cG_0$ is a rank 1 torsion free sheaf as the general fibers of $f$ are connected, and $\cA\subseteq \cG_{0}$.

We now define for each $k$, $\sF_k$ to be the reflexive hull of the intersection 
\[\cG_k|_{Y\setminus \Lambda}\cap \sE_k,\]
where the intersection happens inside $\gr^F_\bullet \cM|_{Y\setminus \Lambda}$. Since $\Lambda$ is of codimension at least 2, $\sF_k$ is a reflexive sheaf on $Y$ for each $k$. In particular, 
\[\sF_\bullet|_{Y\setminus \Lambda}\subseteq \sE_\bullet\]
as $\sA^\bullet_{Y\setminus \Lambda}\big(-\log (E\setminus \Lambda)\big)$-modules.

Furthermore, since $\cG_m|_{U_s}=0$ for all $m\gg0$ by Lemma \ref{lm:coherence}, one can use the arguments in the proof of \cite[Proposition 2.14]{PS17} and conclude that the Higgs field of $\sF_\bullet$ (induced from the Higgs field of $\sE_\bullet$ with logarithmic poles along $E\setminus \Lambda$) only have poles along $D_f$, that is $\sF_\bullet$ is a graded $\sA^\bullet_Y(-\log D_f)$-module. Meanwhile, applying the arguments in the proof of \cite[Proposition 2.15]{PS17}, we also obtain that $\sF_0$ is a line bundle and $\cA(-D_f)\subseteq \sF_{0}$ with the help of \eqref{eq:inc111}.
\end{proof}

Note that the Higgs bundle $\sE_\bullet$ constructed in the previous proposition only support on the base space outside of a codimension at least two subset. Now we are ready to extend it to the whole base space, up to a birational base change, as stated in Theorem \ref{thm:refineHiggs}.

\begin{proof}[Proof of Theorem \ref{thm:refineHiggs}]
We apply Theorem \ref{thm:ggrefine}, and then after replacing $V$ by a birational model, we obtain a compactification $Y$ of $V$ together with a family $f\colon(X,D)\to Y$ and an ample line bundle $\cL$ on $Y$ so that 
\begin{enumerate} [label=(\alph*$'$)]
    \item $E=Y\setminus V$ is simple normal crossing;
    \item the coefficients of $D$ is in $[0,1)$ and $(X,D)$ is log smooth over $V$;
    \item $f^*E+D$ is normal crossing;
    \item $f_*(\omega^m_{X/Y}(mD)\otimes \cL(E)^{-m})$ is generically globally generated for all $m$ sufficiently large and divisible. 
 \end{enumerate}
By (d$'$), the assumption \eqref{as:sections} is satisfied. We then choose a  section 
\[0\neq s\in H^0(X, \omega^m_{X/Y}(mD)\otimes f^*\cL(E)^{-m})).\]
Applying Proposition \ref{prop:constructionHiggs} for the family $f$ and the line bundle $\cL(E)$, one obtains a graded $\sA^\bullet_Y(-\log D)$-module $\sF_\bullet$ with 
\[\cL\subseteq \sF_0,\]
and a reduced divisor $S$ and a geometric VHS defined over $Y\setminus E+S$ so that 
\[\sF_\bullet|_{Y\setminus \Lambda}\subseteq\sE_\bullet\]
where $\Lambda$ is the non-normal crossing locus of $E+S$ and $\sE_\bullet$ is the Higgs bundle defined away from $Z$ associated to the Deligne extesnion with eigenvalues in $[0,1)$ of the VHS. The only problem now is that $\sE_\bullet$ is not globally defined. To fix the problem, we do the following reduction.

One notices that $\sE_\bullet$ remains the same in codimension 1 after adding components to $S$. Hence, we can assume that $S\ge \Delta_\pi$, where $\Delta_\pi$ is the branched divisor of the reduction map $\pi$ in the proof of Theorem  \ref{thm:ggrefine}. We now apply Proposition \ref{prop:liftsections}. To simplify notations, we replace   
$V$ by $\mu^{-1}V$, $\cL$ by $\mu^*\cL$ (hence $\cL$ becomes nef and big), $f$ by $\tilde f$, $E$ by $\mu^*E$ and $S$ by $\mu^{-1}S$. Then we apply Proposition \ref{prop:constructionHiggs} again for the new $f$ and $\cL(E)$ away from the codimension at least 2 subset $T$. By Proposition \ref{prop:liftsections}, we conclude that $s$ and $D$ do not change over $f^{-1}(V\setminus S)$. Hence, the Hodge module $\cM$ and the VHS remain the same over $V\setminus S$. But, $E+S$ becomes normal crossing and hence $\Lambda$ is empty this time. Therefore, we have obtained the required $\sF_\bullet$ and $\sE_\bullet$ with the required properties away from $T$. Since $T$ is of codimension at least $2$ and $\sE_\bullet$ is globally defined, we can replace $\sF_\bullet$ by its reflexive hull globally and the required properties are still true. Now, the proof is accomplished.   
\end{proof}
\begin{remark}\label{rm:main}
In \cite{VZ02,VZ03}, the authors used a weak decomposition theorem of log VHS given by rank one unitary representations under cyclic coverings together with a classic result of Steenbrink \cite[Theorem 2.18]{Ste75} to construct Higgs sheaves. The result gives geometric descriptions of Deligne canonical extensions of geometric VHS in codimension one by using relative log forms. Because of the use of the decomposition theorem and relative forms, their construction requires the canonical bundles of fibers being semi-ample. 
The base-point freeness of pluricanonical bundles for fibers is also needed in proving Brody hyperbolicity of moduli stack of canonicallly polarized manifolds; see \cite[Lemma 5.4]{VZ03}.

It was due to Popa and Schnell \cite{PS17} that using the tautological sections of cyclic coverings and Hodge modules to give an alternative construction of Higgs sheaves. The use of cyclic coverings helps avoid fiberwisely analyzing of base locus of pluricanonical bundles. 
Similar constructions are also used in \cite{PS14} and \cite{W17b}.
\end{remark}

\section{Proof of the main Theorem}\label{sect:proofmain}
In this section, we prove our main theorem and its application: Theorem \ref{thm:main} and Corollary \ref{thm:mgn}.
\begin{proof}[Proof of Theorem \ref{thm:main} (1)] 
By Lemma \ref{L:compare var}, after perturbing the coefficient of $D_U$, we are allowed to assume that $f_V\colon (U, D_U)\to V$ is a log smooth family of log general type klt pairs with maximal variation. Since birational modification of $V$ is harmless to log general typeness of $V$, after modifying $f_V$ birationally, we can assume that we have the outputs of Theorem \ref{thm:refineHiggs}.

The Higgs sheaf inclusion $\sF_\bullet\subseteq \sE_\bullet$ gives a commutative diagram for each $k\ge0$
\[\begin{tikzcd}
  \sF_k \arrow[r,"\theta_k"] \arrow[d, hook] & \sF_{k+1}\otimes\Omega^1_Y(\log E) \arrow[d, hook]\\
  \sE_k \arrow[r] &\sE_{k+1}\otimes\Omega^1_Y(\log E+S)
\end{tikzcd}\]
We write $\cK_k=\textup{ker }\theta_k$. By \cite[Theorem A]{PW}, the above diagram tells that $\cK_k^*$ is weakly positive for $k\ge 0$. 

We then follow the strategy in \cite{PS17}, which is essentially due to Viehweg and Zuo \cite{VZ02}. The key point is the construction of the so-called Viehweg-Zuo sheaf. 

To construct it in our situation, we look at a chain of morphisms of $\cO_X$-modules induced by $\theta_\bullet$
\[\cL\hookrightarrow \sF_0\stackrel{\theta_0}{\rarr} \sF_1\otimes \Omega^1_X(\log E)\stackrel{\theta_1\circ \textup{id}}{\rarr} \sF_2\otimes \big(\Omega^1_X(\log E)\big)^{\otimes2}\rarr \cdots .\]
We consider the induced morphism $\phi_p\colon\cL\to\sF_p\otimes \big(\Omega^1_X(\log D_f)\big)^{\otimes p}$ for $p\ge1$. Since $\cL$ is a line bundle and $\sF_p$ is reflexive , this morphism is either injective or $0$.  We pick $p$ to be the smallest integer so that $\phi_p$ is injective and $\phi_{p+1}=0$. Since $\phi_{p+1}=0$, we know $\tau_p$ factors through
\[\cL\hookrightarrow \cK_p \otimes \big(\Omega^1_X(\log E)\big)^{\otimes p}.\]
Hence, we have a nontrivial morphism $\cL\otimes \cK_p^*\rarr \big(\Omega^1_X(\log E)\big)^{\otimes p}.$
Since $\cL$ is big and $\cK_p^*$ is weakly positive, $\cL\otimes \cK_p^*$ is a big torsion-free sheaf. The Viehweg-Zuo sheaf $\cH$ is defined to be the image of  $\cL\otimes \cK_p^*$ inside $ \big(\Omega^1_X(\log D_f)\big)^{\otimes p}$, which is also a big torsion-free sheaf.

Applying a theorem of Campana-P\u aun \cite[Theorem 7.6]{CP15} (see also \cite[Theorem 19.1]{PS17}), we conclude that $\omega_Y(E)$ is of log general type. 
\end{proof}

Let $Y$ be a smooth variety with $\Gamma$ a simple normal crossing divisor on $Y$, and a geometric VHS defined on $Y\setminus \Gamma$. Assume that $\cG_\bullet$ is a graded $\sA_Y^\bullet(-\log \Gamma)$-module.
The graded $\sA_Y^\bullet(-\log \Gamma)$-module structure induces the Kodaira-Spencer map 
\[\rho_k^q\colon\textup{Sym}^q(\shTA_Y(-\log~\Gamma))\rarr \mathcal{H}\textup{om}(\cG_k,\cG_{k+q})\]
for every $k$ and $q>0$.

Recently, Deng (\cite[Theorem C]{Den18b}) proves that for Deligne extensions of geometric VHS, the Kodaira-Spencer maps have some Torelli-type property when some positive requirement are satisfied. We present here a slightly stronger version using our terminology, which can be proved by essentially the same arguments as in $loc.$ $cit.$ 
\begin{theorem}[Deng]\label{th:Deng}
Let $Y$ be a smooth projective variety with $\Gamma$ a simple normal crossing divisor on $Y$. Assume that $\sE_\bullet$ is the Higgs bundle associated to a Deligne extension of a VHS defined over $Y\setminus \Gamma$ with non-negative eigenvalues along every component of $\Gamma$ and $\cL$ is a nef and big divisor on Y. If $\cL\subseteq \sE_k$ for some $k$, then the composed sheaf morphism 
\[\tau\colon\shTA_Y(-\log~\Gamma)\stackrel{\rho_k^1}{\rarr} \sE_k^{-1}\otimes\sE_{k+1}\to \cL^{-1}\otimes \sE_{k+1}\]
is injective.
\end{theorem}
Deng's proof relies on the asymptotic behavior of Hodge metric at infinity (cf. \cite[Lemma 3.2]{PTW18}), which is still the same when the eigenvalues along $\Gamma$ are only positive. The proof originally deals with the case when $\cL$ is contained in the initial term of the Higgs bundle. When $\cL\subseteq \sE_k$ for some arbitrary $k$, one needs to consider additionally the term $\theta_{k+1}\wedge \theta^*_{k+1}$ when applying Griffiths curvature formula, where $\theta_k$ is the $k$-th (twisted) Higgs field and $\theta^*_k$ is its adjoint with respect to the (twisted) Hodge metric. But $\theta_{k+1}\wedge \theta^*_{k+1}$ is semi-positive, which is harmless to Deng's arguments; see \cite[Proof of Theorem C]{Den18b} for details. Therefore, Deng's theorem holds in the more general situation as above.  
 
Applying the above theorem to the Higgs bundle in Theorem \ref{thm:refineHiggs}, one immediately obtains the following corollary.
\begin{coro}\label{cor:injtau}
In the situation of Theorem \ref{thm:refineHiggs}, the induced composed morphism 
\[\tau_1\colon \shTA_Y(-\log E)\stackrel{}{\rarr}\mathcal{H}\textup{om}(\sF_0,\sF_1)\to \sL^{-1}\otimes \sE_1\]
is injective.
\end{coro}

\begin{proof}[Proof of Theorem \ref{thm:main} (2)]
As in the proof of the first statement, we are free to assume that $f_V\colon (U, D_U)\to V$ is a log smooth family of log general type klt pairs with maximal variation. We assume on the contrary $\gamma\colon \bbC\to V$ is a holomorphic map into the base $V$ with Zariski-dense image. Using the lifting property of entire curves with Zariski-dense images through birantional morphisms (see for instance \cite[Lemma 3.9]{PTW18}), we can assume that we are in the situation of Theorem \ref{thm:refineHiggs}. The differential map of $\gamma$ induces a composed morphism
\[\tau_{(\gamma,1)}\colon \shTA_\bbC \to  \gamma^*(\shTA_Y(-\log E))\stackrel{\gamma^*(\tau_1)}{\rarr} \gamma^*(\cL^{-1}\otimes \sE_1).\]
From Corollary \ref{cor:injtau}, $\tau_1$ is injective. Since $\gamma$ has dense image, $\gamma^{*}(\tau_1)$ is also injective (see for instance \cite[Lemma 3.13]{PTW18}). Therefore, $\tau_{(\gamma,1)}$ is also injective. But this contradicts to \cite[Proposition 3.5]{PTW18} (notice that the output of Theorem \ref{thm:refineHiggs} is the same as that of \cite[Proposition 2.5]{PTW18}). To be more precise, the injectivity of $\tau_{(\gamma,1)}$ together with the positivity of $\cL$ would give a negative singular metric on $\bbC$ contradicting the Ahlfors-Schwarz Lemma; see the proof of \cite[Proposition 3.5]{PTW18} for details.
\end{proof}

\subsection*{Remarks on Kobayashi hyperbolicity}\label{scn:kobayashi}
In \cite{Den18, Den18b}, Deng proves the Kobayashi hyperbolicity of bases of efffectively parametrized families of minimal projective manifolds of general type and pseudo Kobayashi hyperbolicity of bases of families of projective manifolds with semi-ample canonical bundles and maximal variation. The latter further implies Brody hyperbolicity of the moduli of polarized manifolds.  We briefly explain here how Deng's approach could apply to our situation.

The key point is the construction of a negative metric from a geometric VHS containing some positive line bundle, when assuming the generic Torelli property for the Kodaira-Spencer map of the VHS. This is done as in \cite[\S 3]{Den18}; please see $loc.$ $cit.$ for definitions and details, including an explanation on how this uses the Griffiths curvature formula for (twisted) Hodge bundles and the work of Schumacher, To-Yeung and Berndtsson-P{\u a}un-Wang \cite{Sch12, Sch17, TY15, BPW17}; see also \cite{Den18b}.

Recall first that a singular metric $h$ on a vector bundle is said to have semi-negative curvature if the function $\log \|s\|_h$ is plurisubharmonic for every local section $s$ of the vector bundle; cf. e.g. \cite[Definition 18.1]{HPS}. From the output of Theorem \ref{thm:refineHiggs}, applying Corollary \ref{cor:injtau} and \cite[Theorem 3.8]{Den18} yields:

\begin{prop}\label{thm:negcurmetric}
In the situation of Theorem \ref{thm:refineHiggs}, there exists a singular metric $F$ on $\shTA_Y(-\log E)$, with semi-negative curvature, such that on the locus where $F$ is smooth its holomorphic sectional curvature is bounded from above by a negative constant. 
\end{prop}

Recall that for a quasi-projective variety $V$, its exceptional locus is defined as 
\[\textup{Exc}(V)\coloneqq\overline{(\bigcup\gamma(\bbC))},\]
where the union is taken over all non-constant holomorphic maps $\gamma\colon \bbC\to V$ and the closure is in the Zariski topology. Standard facts about the non-degeneracy of the Kobayashi-Royden pseudo-metric imply then the following  statement, which in particular gives another proof of (a slightly stronger version of) Theorem \ref{thm:main} (2):

\begin{coro}\label{cor:another}
Let $f_V\colon U \to V$ be a log smooth family of log general type klt pairs, having maximal variation.
Then $V$ is Kobayashi hyperbolic away from a proper subvariety $Z$ of $V$ (that is, the Kobayashi-Royden pseudo-metric is nondegenerate away from $Z$). In particular ${\rm Exc}(V)$ is a proper 
subset of $V$.
\end{coro}


We conclude by noting that Brunebarbe and Cadorel \cite[Theorem 1.6]{BC} have obtained a criterion for a compact complex manifold to be of log general type, in terms of the existence of a semi-negatively curved singular metric on 
$\shTA_Y(-\log E)$,
with negative holomorphic sectional curvature bounded above generically by a negative constant. 
Applying this and Proposition \ref{thm:negcurmetric} gives an alternative way to prove Theorem \ref{thm:main} (1).

\bibliographystyle{amsalpha}
\bibliography{mybib}
\end{document}